\documentclass[12pt,a4paper]{amsart}
\usepackage{amssymb}
\usepackage[T1]{fontenc}
\usepackage{Alegreya}
\usepackage[libertine]{newtxmath}
\usepackage{mathrsfs}

\usepackage[all]{xy}
\usepackage[headings]{fullpage}
\usepackage{paralist}
\usepackage{xspace}

\usepackage{color}
\usepackage[normalem]{ulem}

\usepackage[colorlinks=true,linkcolor=blue, citecolor=blue, urlcolor=blue,%
pagebackref=true]{hyperref}
\makeatletter

\@addtoreset{equation}{section}
\def\theequation{\thesection.\@arabic \c@equation}

\def\@citecolor{blue}
\def\@linkcolor{blue}
\def\@urlcolor{blue}
\def\theenumi{\@alph\c@enumi}
\def\theenumii{\@roman\c@enumii}
\def\p@enumii{}

\makeatother

\swapnumbers
\theoremstyle{plain}
\newtheorem{theorem}[equation]{Theorem}
\newtheorem{lemma}[equation]{Lemma}
\newtheorem{corollary}[equation]{Corollary}
\newtheorem{proposition}[equation]{Proposition}

\theoremstyle{definition}

\newtheorem{conjecture}[equation]{Conjecture}
\newtheorem{remark}[equation]{Remark}

\newtheorem{example}[equation]{Example}

\newtheorem{definition}[equation]{Definition}

\newtheorem{notation}[equation]{Notation}
\newenvironment{notationbox}[1][]{%
    \begin{notation}[#1]\pushQED{\qed}}{\popQED \end{notation}}

\newtheorem{discussion}[equation]{Discussion}
\newenvironment{discussionbox}[1][]{%
    \begin{discussion}[#1]\pushQED{\qed}}{\popQED \end{discussion}}

\newtheorem{observation}[equation]{Observation}

\newtheorem{construction}[equation]{Construction}

\newenvironment{steps}%
    {\setcounter{step}{0}}{} \newcounter{step}
\newcommand{\proofstep}{\par\refstepcounter{step}Step~\thestep.\space\ignorespaces}

\newcommand{\calD}{\mathcal D}

\newcommand{\bbF}{\mathbb F}

\newcommand{\calP}{\mathcal P}
\newcommand{\frakp}{{\mathfrak p}}

\newcommand{\frakq}{{\mathfrak q}}

\DeclareMathOperator{\Sym}{Sym}

\newcommand{\naturals}{\mathbb{N}}
\newcommand{\ints}{\mathbb{Z}}

\def\to{\longrightarrow}
\DeclareMathOperator{\rank}{rk}

\DeclareMathOperator{\height}{ht}

\DeclareMathOperator{\In}{in}

\DeclareMathOperator{\charact}{char}

\newcommand{\define}[1]{\emph{#1}}
\newcommand{\minus}{\ensuremath{\smallsetminus}}
\DeclareMathOperator{\Ass}{Ass}

\DeclareMathOperator{\Frac}{Frac}

\DeclareMathOperator{\GL}{GL}

\newcommand{\hilbertIdeal}{\mathfrak{h}}

\newif\ifreadkumminibib
\readkumminibibfalse
\title{On polynomial invariant rings in modular invariant theory}
\author{Manoj Kummini}
\address{Chennai Mathematical Institute, Siruseri, Tamilnadu 603103. India}
\email{mkummini@cmi.ac.in}

\author{Mandira Mondal}
\address{Chennai Mathematical Institute, Siruseri, Tamilnadu 603103. India
\small{and} 
Indian Institute of Technology, Jodhpur, Rajasthan 342030. India }
\email{mandira@iitj.ac.in}

\thanks{Both authors were partially supported
by an Infosys Foundation fellowship.}
\thanks{MM thanks the National Board of Higher Mathematics and the
Department of Science and Technology, Government of India for the grant
DST/INSPIRE/04/2021/001549.}
\begin{document}
\begin{abstract}
Let $\Bbbk$ be a field of characteristic $p>0$,
$V$ a finite-dimensional $\Bbbk$-vector-space,
and $G$ a finite $p$-group acting $\Bbbk$-linearly on $V$.
Let $S = \Sym V^*$.
Confirming a conjecture of Shank-Wehlau-Broer, we show that
if $S^G$ is a direct summand of $S$, then $S^G$ is a polynomial ring, in
the following cases:
\begin{enumerate}

\item
$\Bbbk = \bbF_p$ and $\dim_\Bbbk V = 4$; or

\item
$|G| = p^3$.
\end{enumerate}
In order to prove the above result,
we also show that
if $\dim_\Bbbk V^G \geq \dim_\Bbbk V - 2$, then the
Hilbert ideal $\hilbertIdeal_{G,S}$ is a complete intersection.
\end{abstract}

\maketitle

\section{Introduction}
\label{sec:intro}

Let $\Bbbk$ be a field and $V$ a finite-dimensional $\Bbbk$-vector-space.
Let $G\subseteq \GL(V)$ be a finite group.
Then the action of $G$ on $V$ induces an action on $V^*$ which extends to
an action by graded $\Bbbk$-algebra automorphisms of
$S :=\Sym V^*$.
By $S^G$ we denote the ring of invariants.

Suppose that the characteristic of $\Bbbk$ is $p>0$ and that
$G$ is a finite $p$-group.
Various authors have studied necessary and/or sufficient conditions
for the invariant ring $S^G$ to be a polynomial ring; see, e.g.
\cite{ShankWehlauTransferModular1999},
\cite{BroerDirectSummandProperty2005},
\cite{BroerHypersurfaces2006},
\cite{BroerInvThyAbelianTransvGps2010},
\cite{KumminiMondalHilbIdeals20},
\cite{BraunPolynomialringofInvariants2022}, and references therein.

If $S^G$ is a polynomial ring then $S$ is a free $S^G$-module, so $S^G$
is a direct summand of $S$. (See Section~\ref{sec:summand} for (well-known)
remarks about direct summands.) The converse is not in general true; in the
non-modular case, every invariant ring is a direct summand of the
respective symmetric algebra. However, for
actions of finite $p$-groups in characteristic $p>0$, the
Shank-Wehlau-Broer conjecture
(R.~J.~Shank and
D.~L.~Wehlau~\cite[Conjecture~1.1]{ShankWehlauTransferModular1999},
reformulated by
A.~Broer~\cite[Corollary~4, p.~14]{BroerDirectSummandProperty2005})
is the following:
\begin{conjecture}
\label{conjecture:SWB}
$S^G$ is a polynomial ring if it is a direct summand of $S$ as an
$S^G$-module.
\end{conjecture}

In this paper we prove a few results in the direction of settling
Conjecture~\ref{conjecture:SWB}.
Let $W \subseteq V^G$ be a subspace.
We recall the following definition
from~\cite[p.~406]{BroerInvThyAbelianTransvGps2010}.
The \define{Hilbert ideal of $G$ in $S$ relative to $W$}, denoted
$\hilbertIdeal_{G,S,W}$,
is the $S$-ideal $((W^\perp S) \cap S^G)S$,
where $W^\perp = \ker (V^* \to W^*)$,
i.e., the subspace of linear forms on $V$ that vanish on $W$.
(We will often call it a \define{relative Hilbert ideal}.)
The \define{Hilbert ideal} $\hilbertIdeal_{G,S }$ is the $S$-ideal
generated by the invariant polynomials of positive degree, i.e., by
$(S^G)_+$.
Note that $\hilbertIdeal_{G,S}$ is $\hilbertIdeal_{G,S,0}$.
If $S^G$ is a direct summand, then every generating set
of $\hilbertIdeal_{G,S}$ (as an $S$-ideal) consisting of homogeneous
invariant elements generates $S^G$ as a $\Bbbk$-algebra;
see~\cite[Theorem~1.6.3]{BensonInvBook1993}
or~\cite[Theorem~2.2.10]{DerksenKemperComputationalInvariantThy2ed2015} for
a proof. Hence if $S^G$ is a direct summand and
$\hilbertIdeal_{G,S }$ is a complete intersection, then $S^G$ is a
polynomial ring. Conversely, if $S^G$ is a polynomial ring then it is a
direct summand of $S$ and
$\hilbertIdeal_{G,S }$ is a complete intersection.

We now state our results about Conjecture~\ref{conjecture:SWB}
when $\dim_\Bbbk V$ is small.
If $\dim_\Bbbk V = 2$, then
there exist $\Bbbk$-linearly independent linear
forms $l_1, l_2 \in S$ such that
$l_1$ and the $G$-orbit product of $l_2$ are in $S^G$;
now apply (b) $\implies$ (a) of
\cite[Theorem~3.9.4]{DerksenKemperComputationalInvariantThy2ed2015}
 to see that $S^G$ is a polynomial ring.
In higher dimensions, we prove the following:

\begin{theorem}
\label{theorem:largeRankVG}
If $\dim V^G \geq \dim V - 2$, then $\hilbertIdeal_{G,S}$ is a complete
intersection.
\end{theorem}

Since $V^G \neq 0$, we get, as an immediate corollary:
Assume that $\dim_\Bbbk V \leq 3$.
Then $\hilbertIdeal_{G,S }$ is a complete
intersection;
in particular, if $S^G$ is a direct summand of $S$ then it is a polynomial
ring.

\begin{theorem}\label{theorem:SWdim4primefield}
Assume that $\Bbbk=\mathbb{F}_p$ and $\dim_\Bbbk V=4$.
If $S^G$ is a direct summand of $S$ then it is a polynomial ring.
\end{theorem}

When $|G|$ is small, we prove the following:

\begin{theorem}
\label{theorem:smallG}
Assume that $|G | = p^3$. If $S^G$ is a direct summand of $S$, then $S^G$
is a polynomial ring.
\end{theorem}

A new observation of this paper that goes into the proofs of the above
results is the following:
For each $W \subseteq V^G$,
$\hilbertIdeal_{G,S,W} = (\hilbertIdeal_{G,S,W} \cap (\Sym W^\perp ) )S$
(\textit{cf}. Proposition~\ref{proposition:extended}).

The problem of determining when
$S^G$ is a polynomial ring is classical; see~\cite{BensonInvBook1993}
and \cite{KemperLoci2002} for a historical account.
The following is a consequence of~\cite[Theorem~C]{KemperLoci2002}.
Using Proposition~\ref{proposition:extended} and
Corollary~\ref{corollary:freemod} we give an alternate proof.

\begin{theorem}
\label{theorem:singloc}
Write  $R = S^G$.
Suppose that for some $W \subseteq V^G$, $R_\frakp$ is a regular local
ring, where $\frakq$ is the $S$-ideal generated by the $\Bbbk$-subspace
$(V/W)^*$ of $V^*$ and $\frakp = \frakq \cap R$.
Then $R$ is a polynomial ring.
\end{theorem}

This paper is organized as follows.
Required definitions are given in Section~\ref{section:prelimI}.
Theorems~{\ref{theorem:largeRankVG}}
and~\protect{\ref{theorem:singloc}}
are proved in Section~\ref{section:dimTwoThree}.
In Section~\ref{section:compositionseries}, we describe a composition
series of $G$.
Section~\ref{section:dimn} uses the results of
\ref{section:compositionseries} and of~\cite{BroerHypersurfaces2006}
to look at how relative Hilbert ideals and the property of being polynomial
rings change as we move along the composition series; results of this
section are used in Section~\ref{section:hilb4} to prove
Theorem~\ref{theorem:SWdim4primefield}.
Theorem~\ref{theorem:smallG} is proved in Section~\ref{section:smallG}.
Finally, in Section~\ref{section:examples}, we discuss some examples
related to the results in the earlier sections.

\subsection*{Acknowledgements}
We thank Kemper for pointing to us his paper~\cite{KemperLoci2002} and for
sharing~\cite{KemperEtAlDatabase}.
We also thank the referee for a careful reading of the paper and helpful
comments.
The computer algebra systems~\cite{M2} and~\cite {Singular}
provided valuable assistance in studying examples.

\section{Preliminaries}
\label{section:prelimI}

The following notation will apply throughout the paper.

\begin{notationbox}
\label{notationbox:general}
Let $\Bbbk$ be a field and $V$ an $n$-dimensional $\Bbbk$-vector-space.
Let $G\subseteq \GL(V)$ be a finite group.
Let $S = \Sym(V^*)$ and $R = S^G$.
\end{notationbox}

Note that $S$ is an $n$-dimensional polynomial ring
and that $G$ acts on $S$ as degree-preserving $\Bbbk$-algebra
automorphisms.
Let $W\subseteq V^G$ be a subspace. Using a result of Nakajima, one can
describe a generating set of $\hilbertIdeal_{G,S}$ as union of a generating
set of $\hilbertIdeal_{G,S,W}$ and $n-\dim W$ homogeneous
polynomials described below. Thus if the relative Hilbert ideal is a
complete intersection then so is the Hilbert ideal.

\begin{proposition}%
[\protect{\cite[Lemma~2.2 and its proof]{BroerInvThyAbelianTransvGps2010}}]
\label{proposition:Broer22}
Let $W \subseteq V^G$ be a subspace of codimension $s$.
Let $\frakp = W^\perp S \cap R$.
Then
\begin{enumerate}

\item
\label{proposition:Broer22:fs}

there exist homogeneous $f_{s+1}, \ldots, f_n \in R$
 such that
\begin{enumerate}

\item
\label{proposition:Broer22:poly}
$R/\frakp$ is an $(n-s)$-dimensional polynomial ring generated by (the
residue classes of) $f_{s+1}, \ldots, f_n$.

\item
\label{proposition:Broer22:hilb}
$\hilbertIdeal_{G,S } = \hilbertIdeal_{G,S,W } +
(f_{s+1}, \ldots, f_n )S$.

\item
\label{proposition:Broer22:regseq}
$f_{s+1}, \ldots, f_n$ form a regular sequence on $S/\hilbertIdeal_{G,S,W}$.

\end{enumerate}

\item
\label{proposition:Broer22:mingens}
$\mu(\hilbertIdeal_{G,S }) = \mu(\hilbertIdeal_{G,S,W }) + (n-s)$, where
for a homogeneous $S$-ideal $J$, $\mu(J)$ denotes the cardinality of a
minimal generating set. In particular, if
$\hilbertIdeal_{G,S,W }$ is a complete intersection, then
$\hilbertIdeal_{G,S }$ is a complete intersection.
\end{enumerate}
\end{proposition}

Statements
\eqref{proposition:Broer22:fs}\eqref{proposition:Broer22:poly}
and
\eqref{proposition:Broer22:fs}\eqref{proposition:Broer22:hilb}
are from~\cite[Lemma~2.2]{BroerInvThyAbelianTransvGps2010} and its proof.
Statements
\eqref{proposition:Broer22:fs}\eqref{proposition:Broer22:regseq}
and
\eqref{proposition:Broer22:mingens}
will be proved after
Corollary~\ref{corollary:ass}, whose proof will only use
\eqref{proposition:Broer22:fs}\eqref{proposition:Broer22:poly}
and
\eqref{proposition:Broer22:fs}\eqref{proposition:Broer22:hilb}.

\begin{discussionbox}
\label{discussionbox:free}
Let $A \subseteq A'$ be noetherian
$\naturals$-graded algebras with $A_0 = (A')_0 = \Bbbk$.
Suppose that $A'$ is a polynomial ring and that $A'$ is a free $A$-module.
Then $A$ is a polynomial ring. To see this, let $F_\bullet$ be a minimal
graded free resolution of $A/A_+$ as an $A$-module. Then
$F_\bullet \otimes_A A'$ is a graded free resolution of $A'/A_+A'$ as an
$A'$-module (since $A'$ is a free $A$-module). It is minimal since $A_+
\subseteq (A')_+$. Since $A'$ is a polynomial ring,
$F_\bullet \otimes_A A'$ is bounded, so
$F_\bullet$ is bounded. Hence $A$ is a polynomial ring.
See, e.g.,~\cite[6.2.3]{BensonInvBook1993} for details.
\end{discussionbox}

\begin{discussionbox}
\label{discussionbox:terminal}
In order to work in arbitrary characteristic,
we modify the definition of terminal variables
from~\cite{ElmerSezerLocallyFiniteDeriv2016}.
Fix a basis $\{x_1,\ldots, x_n\}$ of $V^*$.
Say that $x_j$ is a \define{terminal} variable if
\begin{enumerate}

\item
$(g-1)x_j \in
\Bbbk \langle x_1, \ldots, x_{j-1 }, x_{j+1 }, \ldots, x_n\rangle$ for each
$g \in G$; and

\item
$\Bbbk \langle x_1, \ldots, x_{j-1 }, x_{j+1 }, \ldots, x_n\rangle$ is a
$G$-stable subspace of $V^*$.

\end{enumerate}
Suppose that $x_j$ is a terminal variable.
Define $\Bbbk$-linear \define{locally finite iterative higher derivations}
$\Delta_{j}^{(l)} : S\longrightarrow S$ as follows.
Let $t$ be a new variable and extend the $G$-action on $S$ to $S[t]$ with
$G$ acting trivially on $t$.
Let $\phi_j$ be the ring map $S\longrightarrow S[t]$ given by
\[
x_i \mapsto
\begin{cases}
x_i & \text{if}\; i\neq j,\\
x_j+t & \text{if}\; i=j.
\end{cases}
\]
We now show
that $\phi_j$ is $G$-equivariant since $x_j$ is a terminal variable.
Let $g \in G$.
If $i \neq j$,
then $gx_i \in \Bbbk \langle x_1, \ldots, x_{j-1 }, x_{j+1 }, \ldots,
x_n\rangle$, so
$\phi_j(gx_i)  = gx_i = g(\phi_j(x_i))$.
On the other hand, since $gx_j = x_j + \alpha$ for some $\alpha \in
\Bbbk \langle x_1, \ldots, x_{j-1 }, x_{j+1 }, \ldots,
x_n\rangle$, it follows that
$\phi_j(gx_j ) = (x_j+\alpha)+t = g(\phi_j(x_j))$.
Hence $\phi_j$ is $G$-equivariant.
Define $\Delta_j^{(l)}$ by $\phi_j(f) = \sum \Delta^{(l)}_j(f)t^l$.
Since $\phi_j$ is $G$-equivariant, we see that
$\Delta_j^{(l)}(f)\in S^{G}$ whenever $f\in S^{G}$.
\end{discussionbox}

The following proposition was proved
in~\cite[Lemma~2.2]{ElmerSezerLocallyFiniteDeriv2016} in the context of
modular invariants. The same proof works even if $|G|$ is non-zero in
$\Bbbk$.

\begin{proposition}
\label{proposition:altsum}
Assume that $x_j$ is a terminal variable. Let $f \in S$. Write
\[
f = f_0 + f_1x_j + \cdots + f_d x_j^d
\]
where the $f_i$ belong to
$\Bbbk[x_1, \ldots, x_{j-1}, x_{j+1}, \ldots, x_n]$.
Then
\[
\sum_{l\geq 0 } (-x_j)^l \Delta_{j}^{(l ) } f = f_0.
\]
\end{proposition}

\begin{proof}
Let $\psi_j : S[t ] \to S$ be the ring map fixing $S$ and sending $t
\mapsto -x_j$. Then
\[
\sum_{l\geq 0 } (-x_j)^l \Delta_{j}^{(l ) } f
=
(\psi_j \circ \phi_j )(f ) = f_0.
\qedhere
\]
\end{proof}

\begin{notationbox}
\label{notationbox:charp}
In addition to Notation~\ref{notationbox:general}
we now assume that
$\charact \Bbbk = p>0$ and that $G$ is a finite $p$-group.
\end{notationbox}

The following lemma is known, but we give a brief sketch of the proof,
since we need the details in later sections. The proof gives a basis $v_1,
\ldots, v_n$ of $V$ such that if $r$ denotes $\dim_\Bbbk V^G$, then
$v_{n-r+1}, \ldots, v_n$ form a basis of $V^G$.
This is used, for example, in the proof of
Proposition~\ref{proposition:largeRankVGRel}.

\begin{lemma}
\label{lemma:triang}
There exists a basis $\{v_1, \ldots, v_n\}$ of $V$
such that
\[
g(v_{i})-v_{i}\in \Bbbk\langle v_{i+1}, \ldots, v_{n}\rangle.
\]
(In other words, the elements of $G$ act on $V$ as
lower-triangular unipotent matrices.)
\end{lemma}

\begin{proof}
It is known that $V^G \neq 0$~\cite[4.0.1]{CampbellWehlauModularInvThy11}.
Let $W$ be any non-zero subspace of $V^G$.
Say $t = \dim W$.
$G$ acts on $V/W$; then, by induction on dimension,
there exist $v_1, \ldots, v_{n-t } \in V$ such
that their images in $V/W$ is a basis of $V/W$ with respect to which every
element of $g$ (for its action on $V/W$) is lower triangular and unipotent.
Let $v_{n-t+1 }, \ldots, v_n$ be a basis of $W$. Then
every element of $G$ is lower-triangular with respect to
$\{v_1, \ldots, v_n\}$ and unipotent.
\end{proof}

\section{The Hilbert ideal relative to a subspace of $V^G$}
\label{section:dimTwoThree}

We first prove that the relative Hilbert ideal is an extended ideal
(Proposition~\ref{proposition:extended}) and
some corollaries.
Theorems~\protect{\ref{theorem:largeRankVG}}
and~\protect{\ref{theorem:singloc}}
would then follow.

\begin{proposition}
\label{proposition:extended}
Let $W \subseteq V^G$.
Write $S' = \Sym W^\perp$.
Then $\hilbertIdeal_{G,S,W} = (\hilbertIdeal_{G,S,W} \cap S' )S$.
\end{proposition}

\begin{proof}
Note that
$\hilbertIdeal_{G,S,W} \supseteq (\hilbertIdeal_{G,S,W} \cap S' )S$.
Therefore we need to show that
$\hilbertIdeal_{G,S,W} \subseteq (\hilbertIdeal_{G,S,W} \cap S' )S$, which
we do by induction on degree.

Write $n-r+1 = \dim_\Bbbk W$.
Let $\{ v_r, \ldots, v_n\}$ be a basis of $W$.
Extend it to a basis $\{v_1, \ldots, v_n \}$ of $V$.
Let $\{x_1,\ldots, x_n\}$
be the basis of $V^*$ dual to $\{v_1, \ldots, v_n \}$.
Then $W^\perp = \Bbbk \langle x_{1}, \ldots, x_{r-1}\rangle$
and $S' = \Bbbk [ x_{1}, \ldots, x_{r-1}]$.

Note that for every $r \leq j \leq n$, $x_j$ is a terminal variable
(Discussion~\ref{discussionbox:terminal}).
Let $F \in W^\perp S \cap S^G$.
Then for every $r \leq j \leq n$ and for every $t \geq 0$,
$\Delta_j^{(t)} F \in W^\perp S \cap S^G$.
Note that for $s_1, s_2 \in S$, $\Delta_j^{(t)} (s_1s_2) =
\sum_{0 \leq t' \leq t}\Delta_j^{(t')} (s_1 )\Delta_j^{(t-t')} (s_2 )$.
Hence for every $F \in \hilbertIdeal_{G,S,W}$ we have
$\Delta_j^{(t)} F \in \hilbertIdeal_{G,S,W}$,
for every $r \leq j \leq n$ and for every $t \geq 0$.
Therefore if $\deg F$ is minimum among that of the non-zero homogeneous
polynomials in $\hilbertIdeal_{G,S,W}$, then $F \in S'$.
Consequently,
we assume that for every homogeneous $F_1 \in
\hilbertIdeal_{G,S,W}$ with $\deg F_1 < \deg F$, $F_1 \in
(\hilbertIdeal_{G,S,W} \cap S' )S$.
Hence
$\Delta_{r}^{(t_{r})} \cdots \Delta_n^{(t_n)} F \in
(\hilbertIdeal_{G,S,W} \cap S' )S$ for every
${(0, \ldots, 0) \neq (t_{r}, \ldots, t_n) \in \naturals^{n-r+1}}$.
Note that
\[
\sum_{(t_{r}, \ldots, t_n) \in \naturals^{n-r+1}}
(-1)^{t_{r}+ \cdots+ t_n}
x_{r}^{(t_{r})} \cdots x_n^{(t_n)}
\Delta_{r}^{(t_{r})} \cdots \Delta_n^{(t_n)} F
\]
is the coefficient of
$x_{r}^{0} \cdots x_n^{0}$ in $F$ when $S$ is identified
with $S'[x_{r}, \ldots, x_n]$.
By repeated application of Proposition~\ref{proposition:altsum}
we see that it belongs to
$\hilbertIdeal_{G,S,W} \cap S'$; hence
$F \in (\hilbertIdeal_{G,S,W} \cap S' )S$.
\end{proof}

Note that $G$ acts on $S'$; however, it is not true that
$\hilbertIdeal_{G,S,W } \cap S' = \hilbertIdeal_{G,S' }$.
See Example~\ref{example:nothGR}.
In the next corollary $\Ass_S (S/\hilbertIdeal_{G,S,W })$ denotes the set
of associated primes of $S/\hilbertIdeal_{G,S,W }$.

\begin{corollary}
\label{corollary:ass}
$S/\hilbertIdeal_{G,S,W }$ is Cohen-Macaulay and
$\Ass_S (S/\hilbertIdeal_{G,S,W }) = \{ W^\perp S\}$.
\end{corollary}

\begin{proof}
Adopt the notation from the proof of
Proposition~\ref{proposition:extended}.
Hence, by Proposition~\ref{proposition:extended},
$x_r, \ldots, x_n$ is a regular sequence on
$S/\hilbertIdeal_{G,S,W }$.
Note that $\dim S/\hilbertIdeal_{G,S,W } =
 \dim_\Bbbk W = n-r+1$,
so $S/\hilbertIdeal_{G,S,W }$ is Cohen-Macaulay.

Since $S$ is a polynomial ring over $S'$, it follows that
$\height (\hilbertIdeal_{G,S,W } \cap S') =
\height \hilbertIdeal_{G,S,W } =
\height W^\perp S = r-1 = \dim S'$.
Hence
$S'/(\hilbertIdeal_{G,S,W } \cap S')$ has a finite filtration with
successive quotients isomorphic to $S'/(x_1, \ldots, x_{r-1})S'$.
Since $S$ is a flat $S'$-algebra,
$S/\hilbertIdeal_{G,S,W }$ has a filtration with
successive quotients isomorphic to $S/(x_1, \ldots, x_{r-1})S$.
Then $\Ass_S (S/\hilbertIdeal_{G,S,W }) =
\{ (x_1, \ldots, x_{r-1})S \} = \{W^\perp S \}$.
\end{proof}

\begin{corollary}
\label{corollary:freemod}
Write $R = S^G$ and $\frakp = W^\perp S \cap R$.
Then $S/\hilbertIdeal_{G,S,W }$ is a free $R/\frakp$-module.
\end{corollary}

\begin{proof}
$R/\frakp$ is a polynomial ring
(Proposition~\ref{proposition:Broer22}\eqref{proposition:Broer22:fs}\eqref{proposition:Broer22:poly}).
By Corollary~\ref{corollary:ass},
$S/\hilbertIdeal_{G,S,W }$ is a finitely generated graded Cohen-Macaulay
$R/\frakp$-module with $\dim S/\hilbertIdeal_{G,S,W } = \dim R/\frakp$.
Now use the Auslander-Buchsbaum formula.
\end{proof}

\begin{proof}[Proof of Proposition~\protect{\ref{proposition:Broer22}}]
\eqref{proposition:Broer22:fs}\eqref{proposition:Broer22:poly} and~\eqref{proposition:Broer22:fs}\eqref{proposition:Broer22:hilb}
are
from~\cite[Lemma~2.2]{BroerInvThyAbelianTransvGps2010} and its proof.

\eqref{proposition:Broer22:fs}\eqref{proposition:Broer22:regseq}:
By Corollary~\ref{corollary:freemod},
$f_1, \ldots, f_s$, which is an $R/\frakp$-regular sequence
is regular on $S/\hilbertIdeal_{G,S,W }$.

\eqref{proposition:Broer22:mingens}:
Note that for every homogeneous $S$-ideal $J$ and homogeneous
$a \in S$ regular on $S/J$, $\mu(J + (a) ) = \mu(J )+1$.
\end{proof}

\begin{proposition}
\label{proposition:largeRankVGRel}
Assume that $\charact \Bbbk = p>0$ and that $G$ is a finite $p$-group.
If $\dim V^G \geq \dim V - 2$, then $\hilbertIdeal_{G,S,V^G}$ is a
complete intersection.
\end{proposition}

\begin{proof}
We will apply Proposition~\ref{proposition:extended} with $W = V^G$.
First, construct a basis $\{v_1, \ldots, v_n \}$ of $V$ following the proof
of Lemma~\ref{lemma:triang} with $W = V^G$.
Let $\{x_1,\ldots, x_n\}$ be
the basis of $V^*$ dual to $\{v_1, \ldots, v_n \}$.

Then $v_i \in V^G$ for all $i \geq 3$.
Therefore  $W^\perp = \Bbbk\langle x_1 \rangle$ or
$W^\perp = \Bbbk\langle x_1, x_2 \rangle$.
Hence $S' = \Bbbk[x_1]$ or $S'  = \Bbbk[x_1, x_2]$.
In both cases $x_1 \in S^G$.
Therefore $x_1 \in \hilbertIdeal_{G,S,V^G} \cap S'$.
Hence $\hilbertIdeal_{G,S,V^G} \cap S'$ is a complete intersection.
By Proposition~\ref{proposition:extended},
$\hilbertIdeal_{G,S,V^G}$ is a complete intersection.
\end{proof}

\begin{proof}[Proof of Theorem~\protect{\ref{theorem:largeRankVG}}]
Apply Propositions~\ref{proposition:largeRankVGRel}
and~\ref{proposition:Broer22}\eqref{proposition:Broer22:mingens}.
\end{proof}

\begin{proof}[Proof of Theorem~\protect{\ref{theorem:singloc}}]

Write $\kappa(\frakp )$ for the residue field at $\frakp$
and $K$ for the fraction field of $R$.
We compute
$S \otimes_R \kappa(\frakp )$ in two different ways:
\[
S \otimes_R \kappa(\frakp )
= (S/\frakp S )
\otimes_{R/\frakp} \kappa(\frakp )
=
(R \minus \frakp )^{-1 } S \otimes_{R_\frakp }
\kappa(\frakp ).
\]

By Corollary~\ref{corollary:freemod},
$S/\hilbertIdeal_{G,S,W }$ is is a free $R/\frakp$-module.
Therefore
\[
\dim_{\kappa(\frakp)} ( S \otimes_R \kappa(\frakp ))=
\dim_{\Bbbk} ( S \otimes_R \Bbbk).
\]
On the other hand, since $R_\frakp$ is a regular local ring,
$(R \minus \frakp )^{-1 } S$ is a free $R_\frakp$-module.
Then
\[
\dim_{\kappa(\frakp)} ( S \otimes_R \kappa(\frakp ))=
\dim_{K} ( S \otimes_R K).
\]
Therefore
$\dim_{K} ( S \otimes_R K) = \dim_{\Bbbk} ( S \otimes_R \Bbbk)$.
Using the Nakayama lemma, we see that $S$ is a free $R$-module.
Therefore, by Discussion~\ref{discussionbox:free}, $R$ is a polynomial ring.
\end{proof}

\section{Direct summands}
\label{sec:summand}

In this section, we collect some (well-known) remarks about
direct summands of rings.
Recall that a subring $R$ of a ring $S$ is a \define{direct summand} if
there exists an $R$-linear map $\psi : S \to R$ such that $\psi \circ i
= \mathrm{id}_{R }$, where $i : R \to S$ is the inclusion map.

\begin{proposition}
\label{proposition:summandPolyCM}
Let $R \subseteq S$ be integral domains. Suppose that
there exists $s_0 \in S$ such that
the $R$-submodule $Rs_0$ of $S$ is a free $R$-module
and is a direct summand of $S$ as an $R$-module.
Then $R$ is a direct summand of $S$.
In particular, if $R$ is an $\naturals$-graded polynomial
ring and $S$ is an $\naturals$-graded Cohen-Macaulay domain that is a
finitely generated $R$-module, then $R$ is a direct summand of $S$.
\end{proposition}

\begin{proof}
Let $h : S \to S$ be the $R$-linear map $s \mapsto ss_0$,
$\pi : S \to Rs_0$ be the $R$-linear projection map and
$j : Rs_0 \to R$ the $R$-linear map $rs_0 \mapsto r$.
Define $\psi = j \circ \pi \circ h$.
Then $\psi \circ i = \mathrm{id}_R$; this proves the first assertion.
To prove the second assertion, note, by the Auslander-Buchsbaum formula,
that $S$ is a free $R$-module. Now choose $s_0 \in S$ that can be extended
to an $R$-basis of $S$.
\end{proof}

\begin{proposition}
\label{proposition:summandSub}
Let $R \subseteq A \subseteq S$ be rings. If $R$ is a direct summand of
$S$, then it is a direct summand of $A$.
\end{proposition}

\begin{proof}
There exists an $R$-linear map $\psi : S \to R$ such that $\psi(r ) = r$
for each $r \in R$. Take $\psi|_A : A \to R$.
\end{proof}

\section{A composition series of $G$}
\label{section:compositionseries}

As mentioned in the introduction, we are interested in the following
situation: $\charact \Bbbk = p > 0$ and $G$ is a $p$-group
generated by elements that act as transvections (i.e. non-diagonalizable
pseudoreflections) on $V$.

\begin{notationbox}
The set $\{v_1, \ldots, v_n \}$ will denote a basis of $V$ as in
Lemma~\ref{lemma:triang}.
We write $\{x_1,\ldots, x_n\}$ for the dual basis of $V^*$ dual to
$\{v_1, \ldots, v_n \}$.
For $g\in G$, define
\[
\beta_g :=\max \{j \mid x_j \; \text{appears in}\; gx_i-x_i
\;\text{for some}\;i\}.
\]
Let $\calP$ denote the set of pseudoreflections of $G$.
Define $\beta_G :=\max \{\beta_g\mid g\in \mathcal{P}\}$.
\end{notationbox}

%

\begin{discussionbox}
\label{discussionbox:betag}
With respect to the ordered basis $\{x_1, \ldots, x_n\}$, elements of $G$
act on
$V^*$ as upper triangular matrices, with $1$s on the diagonal.
Let $g \in \calP$. Then $g$ is given by a matrix of the form
\[
\begin{bmatrix} I_{\beta_g} & A \\ 0 & I_{n-\beta_g} \end{bmatrix}
\]
where $\rank A = 1$. Fix a non-zero column vector 
\[
\begin{bmatrix}
c_1 \\ \vdots \\ c_{\beta_g}
\end{bmatrix}
\]
of $A$. Let $l = \sum_{i=1 }^{\beta_g } c_i x_i$, which is an element of 
$\Bbbk \langle x_1, \ldots, x_{\beta_g} \rangle$.
Hence we have the following:
\begin{enumerate}

\item
For all $j \leq \beta_g$, $g x_j = x_j$; in particular, if $f \in
\Bbbk[x_1, \ldots, x_{\beta_g}] \subseteq S$, then $gf = f$.

\item
For all $j > \beta_g$, $g x_j - x_j \in 
\Bbbk \langle x_1, \ldots, x_{\beta_g} \rangle$, and, therefore, 
is divisible by $l$ in $S$.

\end{enumerate}
On $V$, with respect to the basis 
$\{v_1, \ldots, v_n\}$, $g$ is given by
\[
\begin{bmatrix} I_{\beta_g} & 0 \\ -A^{\text{tr}} & 
I_{n-\beta_g} \end{bmatrix}
\]
Hence $\beta_g = \max \{i \mid g v_i  \neq v_i\}$.
\end{discussionbox}

Note that $\beta_g$ depends on the choice of the ordered basis 
$\{v_1, \ldots, v_n \}$.

\begin{lemma}
  \label{lemma:conjOfTrans}
  Let $g, h \in \calP$.
  \begin{enumerate}
\item
\label{lemma:conjOfTrans:differentbeta}
$ghg^{-1} \in \calP$ and $\beta_{ghg^{-1}} = \beta_h$.

    \item
    \label{lemma:conjOfTrans:samebeta}
    If $\beta_g = \beta_h$ then $gh=hg$.

  \end{enumerate}
\end{lemma}
\begin{proof}
\eqref{lemma:conjOfTrans:differentbeta}:
%
%
If $j > \beta_h$, then $\Bbbk \langle v_j, \dots, v_n \rangle$ is stable
under the action of $G$ and is fixed by $h$,
so $g h g^{-1}$ fixes $v_j$, from which it follows that
$\beta_{g h g^{-1}} \leq \beta_h$. 
Equality now holds by symmetry.

\eqref{lemma:conjOfTrans:samebeta}:
%
%
%
%
Let $m = \beta_g = \beta_h$. Let $1 \leq i \leq n$.
Then, we saw in Discussion~\ref{discussionbox:betag}, 
$gx_i-x_i \in \Bbbk \langle x_1, \ldots, x_m \rangle$.
Similarly $hx_i-x_i \in \Bbbk \langle x_1, \ldots, x_m \rangle$.
Hence, using Discussion~\ref{discussionbox:betag} again, 
we see that $g (hx_i-x_i ) = hx_i-x_i$ and $h (gx_i-x_i ) = gx_i-x_i$.
Therefore
\[
g(h(x_i)) = g(x_i + (hx_i-x_i)) = x_i + (gx_i-x_i ) + (hx_i-x_i)
=h(g(x_i)).
\qedhere
\]
\end{proof}


Using the above lemma, we see that there is a composition series of $G$
consisting of subgroups generated by transvections.

\begin{proposition}
\label{proposition:ourcompositionseries}
$G$ has a composition series
$0  = G_0 \subsetneq G_1 \subsetneq \cdots \subsetneq G_k = G$
such that for each $1 \leq l \leq k$, $G_l$ is a transvection group and
$G_{l}/G_{l-1}$ is isomorphic to $\ints/p\ints$ and is generated by the
residue class of a transvection.
\end{proposition}

\begin{proof}
For $1 \leq l \leq \beta_G$, define $G_l$ to be the group generated by
$\{g \in \calP \mid \beta_g \leq l\}$. By Lemma~\ref{lemma:conjOfTrans}
$G_l$ is normal in $G_{l+1}$ and $G_{l+1}/G_l$ is an abelian group
generated by the residue classes of $\{g \in \calP \mid \beta_g =l+1\}$.
Refine this filtration, if necessary.
\end{proof}

We give the explicit description of such a
filtration for Nakajima groups in Example~\ref{example:NakajimaFiltration}.
Various authors (see,
e.g.,~\cite[Section~1]{BraunPolynomialringofInvariants2022})
have used inertia subgroups of prime
ideals generated by linear subspaces of $V^*$ (the space of homogeneous
linear forms in $S$) to study homological properties of $S^G$. In
Example~\ref{example:notInertia}, we show that our filtration is not
necessarily given by such inertia subgroups.
Moreover, the reason for looking at the above composition series is the
hope that we could prove Conjecture~\ref{conjecture:SWB}
by induction on the order of the group.
When $G$ is a $p$-group and $\charact \Bbbk =p$, if the invariant ring is a
direct summand, then $G$ is generated
by transvections~\cite[Theorem~2(ii)]{BroerDirectSummandProperty2005}.
Therefore we would like the subgroups in the composition series to be
generated by transvections.
Example~\ref{example:notInertia} shows that inertia subgroups need not be
generated by transvections, in general.

We end this section with some observations. A partial converse to the next
proposition will be proved in Section~\ref{section:dimn} (see
Proposition~\ref{proposition:RpolyThenApoly}).

\begin{proposition}
\label{proposition:allSplit}
$0 = G_0 \subsetneq G_1 \subsetneq \cdots \subsetneq G_k = G$ be the
filtration of Proposition \ref{proposition:ourcompositionseries}.
Assume that $S^{G_l}$ is a direct summand of $S$ for each $l$.
Then $S^{G_l}$ is a polynomial ring for each $l$.
\end{proposition}

\begin{proof}
Since $S = S^{G_0}$ is a polynomial ring, we may assume, by induction, that
$S^{G_{l-1 } }$ is a polynomial ring.
Note that $S^{G_l }$ is a direct summand of $S^{G_{l-1 } }$.
Write $A = S^{G_{l-1 } }$ and $R = S^{G_l}$.
Then there exists a homogeneous $a \in A$ such that $A = R[a]$;
see~\cite[Theorem~4, p.~588]{BroerHypersurfaces2006}.
Moreover, there exist $r_1, r_2 \in R$ such that
$R[a] \simeq R[T]/(T^p-r_1T-r_2 )$.
Hence $A$ is a free $R$-module.
Hence $R$ is a polynomial ring, by Discussion~\ref{discussionbox:free}.
\end{proof}

\section{Relative Hilbert ideals at consecutive stages in the composition
series}
\label{section:dimn}

\begin{notation}
Let
$0 = G_0 \subsetneq G_1 \subsetneq \cdots \subsetneq G_k = G$
be the filtration from
Proposition~\ref{proposition:ourcompositionseries}.
Write $G' = G_{k-1}$.
We denote by $\sigma$ a transvection in $G \minus G'$.
Let $R=S^{G}$ and $A=S^{G'}$.
Let $<$ denote the degree-lexicographic monomial order on $S$ with
$x_1<x_{2}<\cdots<x_n$.
\end{notation}

\begin{lemma}
  \label{lemma:dirSummandLeastDeg}
  Suppose that $R$ is a direct summand of $A$.
  Let $d_0=\min\{d\mid A_d\neq R_d\}$.
  \begin{enumerate}

    \item
    \label{lemma:dirSummandLeastDeg:one}
    $\dim_\Bbbk A_{d_0}/R_{d_0}=1$.

    \item
    \label{lemma:dirSummandLeastDeg:gen}
    $A=R[a]$ for all $a\in A_{d_0}\minus R_{d_0}$.

  \end{enumerate}
\end{lemma}

\begin{proof}

\eqref{lemma:dirSummandLeastDeg:one}: By~\cite[Theorem~4,
p.~588]{BroerHypersurfaces2006}, there exists a homogeneous $a_{0} \in A
\minus R$ such that $A=R[a_0]$.  It is immediate that $\deg a_0 = d_0$.
Suppose there exists $a\in A_{d_0}\minus R_{d_0}$ that is
$\Bbbk$-linearly independent of $a_0$ in $A_{d_0}/R_{d_0}$.  Then $a \not
\in R[a_{0}]$ since the degrees of $a$ and $a_0$ are the same.  This proves
the assertion.

\eqref{lemma:dirSummandLeastDeg:gen}:
Suppose that $a\in A_{d_0}\minus R_{d_0}$.
Let $a_0$ be as in~\eqref{lemma:dirSummandLeastDeg:one}.
  Then, by~\eqref{lemma:dirSummandLeastDeg:one},
  $a+\lambda a_0\in R_{d_0}$ for some $\lambda\in\Bbbk^\times$.
  Hence $a_0\in R[a]$, so $A=R[a]$.
\end{proof}

\begin{lemma}
\label{lemma:sigmaFixes}
We may assume $\sigma(x_i)=x_i$ for all $i\neq n$.
\end{lemma}
\begin{proof}
Let $i_\sigma=\text{max}\{i\mid \sigma x_i\neq x_i\}$.
Note that $i_\sigma > \beta_\sigma$.
For $i>\beta_\sigma$, define $\lambda_i$ by $\sigma
x_i-x_i=\lambda_i(\sigma x_{i_\sigma}-x_{i_\sigma})$.
We define a change of coordinates as follows:
\[
x'_i :=
\begin{cases}
x_i, & \text{if}\; i \leq \beta_\sigma,\\
x_i-\lambda_i x_{i_\sigma}, & \text{if}\; i > \beta_\sigma
\;\text{and}\; i\not\in \{i_\sigma, n \},\\
x_n, & \text{if}\; i = i_\sigma,\\
x_{i_\sigma}, & \text{if}\; i = n.\\
\end{cases}
\]

Note that the action of $G$ on $\Bbbk\langle x'_1, \ldots, x'_n \rangle$ is
`upper-triangular'.
Let $\tau \in G$ be a transvection. Then
\[
\max \{j \mid x_j \;
\text{appears in}\; \tau x_i-x_i \;\text{for some}\; i\}
=
\max \{j \mid x'_j \;
\text{appears in}\; \tau x'_i-x'_i \;\text{for some}\; i\}.
\]

We see this as follows. Both sides give the maximum $j$ that appears in a
linear form defining
the reflecting hyperplane of $\tau$.
However the equation of the reflecting
hyperplane of $\tau$ involves only $x_i$ with 
$i \leq \beta_\sigma$. 
Let $\alpha_1x_1 + \cdots + \alpha_{\beta_\sigma}x_{\beta_\sigma}$ be
a linear form that defines the reflecting hyperplane of $\tau$.
Then in the new coordinates 
$\alpha_1x'_1 + \cdots + \alpha_{\beta_\sigma}x'_{\beta_\sigma}$ 
defines the reflecting hyperplane of $\tau$.
Hence $\beta_\tau$ does not change after this change of coordinates.
\end{proof}

\begin{lemma}
  \label{lemma:algGens}
Suppose that $R$ is a direct summand of $A$.
Then there exists a minimal algebra generating set $\{x_1,a_2,\ldots,a_m\}$
of $A$ with $m\geq n$ such that
  \begin{enumerate}
  \item
  \label{lemma:algGens:one}
    there exists a unique $ 2 \leq i_0 \leq m$ such that
    $a_{i_0 } \not \in R$.

  \item
  \label{lemma:algGens:order}
  $x_1 < \In_>(a_2) < \cdots <  \In_>(a_{m})$;

    \item
  \label{lemma:algGens:exactlyonexnpower}
  there exists a unique $2 \leq i_1 \leq m$ such that
  $\In_>(a_{i_1}) = x_n^{\deg a_{i_1}}$.
  \end{enumerate}
Moreover, $\deg a_{i_1 }$ is a power of $p$.
  \end{lemma}

\begin{proof}
By Lemma~\ref{lemma:dirSummandLeastDeg},
 $A = \Bbbk[x_1, a_2, \ldots, a_{m}]$ for some $a_2, \ldots, a_{m}
\in A$ with $m\geq n+1$, such that all but one of them is in $R$.
From this we choose a minimal generating set
which we continue to call
$x_1, a_2, \ldots, a_{m}$, $m \geq n$, after relabelling the elements, if
necessary.
Note that all but one of the elements of the minimal generating set is in
$R$.
This proves~\eqref{lemma:algGens:one}.

Now suppose that $a, a' \in A$ are part of a $\Bbbk$-algebra
generating set $X$ of $A$ with $\In_>(a) = \In_>(a')$ that
satisfies~\eqref{lemma:algGens:one}.
If $a \not \in R$ or $a' \not \in R$, assume, after relabelling if
necessary, that $a \in R$ and $a' \not \in R$.
Note that $(X \cup
\{a' + \lambda a\} ) \minus \{a' \}$ generates $A$ for all $\lambda \in
\Bbbk^\times$.  For a suitable $\lambda \in \Bbbk^\times$, $\In_>(a'+\lambda a) <
\In_>(a)$. Therefore, without loss of generality we can assume
$\In_>x_1<\In_>a_2<\cdots<\In_>a_m$.
This proves~\eqref{lemma:algGens:order}.

To prove~\eqref{lemma:algGens:exactlyonexnpower},
let
\[
i_1 := \min \{ i \mid  \In_>(a_i) = x_n^{\deg a_i} \}.
\]
(This set is non-empty since  $S$ is integral over $A$.)
Now, for all $r \in A$, if $\deg_{x_n}(r) = p^em$
with $p \nmid m$, then $\deg_{x_n}(\Delta_{x_n}^{(m-1)p^e}r) = p^e$.
Hence, if $\deg_{x_n}(a_{i_1})$ is not a power of $p$, we can find an element
$r \in A$ such that $\deg r < \deg a_{i_1}$ and $\In_>(r)  = x_n^{\deg r }$.
Hence there must exist $j < i_1$ such that
$a_j \in A$ and $\In_>(a_j) = x_n^{\deg a_j}$, contradicting the
minimality of $i_1$.
Therefore write $\deg a_{i_1 } = p^{e_1}$.

Now suppose that $j > i_1$ and $\In_>(a_j) = x_n^{\deg a_j }$.
If $p^{e_1 } \nmid \deg a_j$, then applying
$\Delta_{x_n}^{(t)}$ to it for a suitable $t$, we obtain $a \in A$ with
$\In_>(a ) = x_n^{\deg a }$ and $\deg a < \deg (a_{i_1 } )$, which implies
that there
exists $i< i_1$ with $\In_>(a_i) = x_n^{\deg a_i }$, contradicting the
choice of $i_1$. Hence $p^{e_1 } \mid \deg a_j$. We now consider the
different cases.

\underline{Case 1: $i_0 = i_1$}. Then $a_j \in R$.
Assume, for now, that $p^{e_1+1 } \nmid \deg a_j$.
Applying $\Delta_{x_n}^{(t)}$ to it for a suitable $t$, we obtain $a \in R$
with $\In_>(a ) = x_n^{\deg a }$ and $\deg a = \deg (a_{i_1 } )$.
Hence $a$ belongs to the subalgebra of $A$ generated by
$\{a_i \mid 1 \leq i \leq m , \deg a_i \leq \deg (a_{i_1 }) \}$.
However, since $i_0 = i_1$ and $\In_>(a ) = x_n^{\deg a_{i_1}}$
it follows that $a = \lambda a_{i_1 } + a'$ for some $\lambda \in
\Bbbk^\times$ and $a' \in R$, contradicting the fact that $a \in R$.
Hence $p^{e_1+1 } \mid \deg a_j$.
Then we may replace $a_j$ by
\[
a_j -
\lambda \left(\prod_\sigma a_{i_1}\right )^{\frac{\deg a_j}{p\deg a_{i_1}}}
\]
for a suitable $\lambda \in \Bbbk^\times$.
This may violate~\eqref{lemma:algGens:order}, but the steps in its proof
may be repeated to ensure that~\eqref{lemma:algGens:order} is satisfied.

\underline{Case 2: $i_0 \neq i_1$}. Then $a_{i_1} \in R$.
We may replace $a_j$ by
\[
a_j - \lambda a_{i_1}^{\frac{\deg a_j}{\deg a_{i_1}}}
\]
for a suitable $\lambda \in \Bbbk^\times$.
Again, steps in the proof of~\eqref{lemma:algGens:order}
may be repeated to ensure that~\eqref{lemma:algGens:order} is satisfied.
This completes the proof of the lemma.
\end{proof}

\begin{lemma}
\label{lemma:secondgen}
Adopt the notation from Lemma~\ref{lemma:algGens}.
Assume that $\beta_\sigma > 1$.
Then $a_2 \in \Bbbk[x_1, \ldots, x_{n-1}]^{G}$.
\end{lemma}

\begin{proof}
Since $A \cap \Bbbk[x_1, \ldots, x_{n-1}] \subseteq R$,
it suffices to show that
$a_2 \in \Bbbk[x_1, \ldots, x_{n-1}]$.
By way of contradiction assume that ${x_n}$ divides some terms of $a_2$.
Then in every such term, the exponent of ${x_n}$ must be a power of $p$; for
otherwise, by taking $\Delta_{x_n}^{(t)}a_2$ for a suitable $t$, we get a
(non-zero)
element of $R \minus \Bbbk[x_1 ] $, contradicting the minimality of
$\In_>(a_2)$. Using the same argument, we see that $a_2$ must be of the
form
\[
a_{2,0 } + \sum_{i \geq 0 } \lambda_i x_1^{\deg a_2 - p^i} {x_n}^{p^i}
\]
with $a_{2,0 }  \in \Bbbk[x_1, \ldots, x_{n-1}]$ and $\lambda_i \in \Bbbk$.
Hence
\[
(\sigma -1  ) a_2 =
(\sigma -1  )a_{2,0 } + \sum_{i \geq 0 } \lambda_i x_1^{\deg a_2 - p^i}
(\sigma -1  ){x_n}^{p^i}.
\]
By Lemma~\ref{lemma:sigmaFixes}, $(\sigma-1) a_{2,0 } = 0$.
Note that $(\sigma -1)x_n \not \in \Bbbk\langle x_1 \rangle$.
Therefore, unless $\lambda_i = 0$ for all $i \geq 0$,
$(\sigma-1)a_2$ would be a (non-zero) element of $A \minus \Bbbk[x_1]$ with
$\In_>((\sigma-1)a_2 ) < \In_>(a_2)$, contradicting the choice of $a_2$.
Hence $\lambda_i = 0$ for all $i \geq 0$.
\end{proof}

\begin{lemma}\label{lemma:a3notinR}
Adopt the notation of Lemma~\ref{lemma:algGens}.
Suppose that $R$ is a direct summand of $A$.
Let $W = \Bbbk \langle v_n \rangle$.
If $i_0 = i_1$, then
\[
\hilbertIdeal_{G', S, W} =
\hilbertIdeal_{G, S, W} =
(x_1,a_2, \ldots,a_{i_0-1}, a_{i_0+1}, \ldots,a_m)S.
\]
\end{lemma}

\begin{proof}
Write $\Lambda  = \{ x_1,a_2, \ldots,a_{i_0-1}, a_{i_0+1}, \ldots,a_m\}$.
Let $a \in W^\perp S \cap A$ be a homogeneous element. Write
$a = \Phi(x_1, a_2, \ldots, a_m )$
where $\Phi$ is a polynomial in $m$ variables
$\xi_1, \ldots, \xi_m$.
Since
$\Lambda \subseteq W^\perp S$, it follows that no term of
$\Phi(\xi_1, \ldots, \xi_m)$
is a power of $\xi_{i_0 }$.
Hence $a \in \Lambda S$, and, therefore,
$\hilbertIdeal_{G', S, W} \subseteq \Lambda S$.
Moreover, $\Lambda \subseteq W^\perp S \cap
R$, so we get
\[
\hilbertIdeal_{G', S, W} \subseteq \Lambda S
\subseteq
\hilbertIdeal_{G, S, W}
\subseteq
\hilbertIdeal_{G', S, W},
\]
proving the lemma.
\end{proof}

As an application, we prove a partial converse to
Proposition~\ref{proposition:allSplit}.

\begin{proposition}
\label{proposition:RpolyThenApoly}
Adopt the notation of Lemma~\ref{lemma:algGens}.
Suppose that $i_0 = i_1$ and that $R$ is a polynomial ring.
Then $A$ is a polynomial ring.
\end{proposition}

\begin{proof}
Note that $R$ is a polynomial ring if and only if $S$ is a free $R$-module.
The `only if' direction follows from the Auslander-Buchsbaum formula;
see Discussion~\ref{discussionbox:free} for the `if' direction.
$R$ is a polynomial ring if and only if
$\dim_\Bbbk S/\hilbertIdeal_{G,S }  = |G|$.
Similarly $A$ is a polynomial ring if and only if
$\dim_\Bbbk S/\hilbertIdeal_{G',S }  = |G'|$.
Let $W = \Bbbk \langle v_n \rangle$.

\begin{steps}

\proofstep
$\min \{ \deg r \mid r \in R \minus W^\perp S\}=p \deg a_{i_0}$.
Proof: Let $r \in R \minus W^\perp S$ be of the smallest degree.
Then $\In_>(r ) = x_n^{\deg r }$.
By considering $\Delta_{n}^{(t ) } r, t \geq 1$,
we see that $\deg r$ is a power of $p$.
By the definition of $i_1$, $\deg r \geq \deg a_{i_1}$.
Since $i_0 = i_1$, $r \neq a_{i_1}$.
Hence $\deg r > \deg a_{i_1}$.
Since $\deg a_{i_1 }$ is a power of $p$, it follows that
$\deg r \geq p \deg a_{i_1} = p \deg a_{i_0}$.
On the other hand $\prod_\sigma a_{i_0} \in R \minus W^\perp S$, so
$\deg r \leq p \deg a_{i_0}$.

\proofstep
Let $r_0 \in R \minus W^\perp S$ be such that $\deg r_0 = p \deg a_{i_0}$.
Since
$\Ass_S (S/\hilbertIdeal_{G,S,W}) = W^\perp S$
(Corollary~\ref{corollary:ass})
$a_{i_0}$ and $r_0$ are non-zero-divisors on
$S/\hilbertIdeal_{G,S,W}$.

\proofstep
Note that $R/(W^\perp S \cap R )$ is a $1$-dimensional polynomial ring
(Proposition~\ref{proposition:Broer22}).
Hence $R/(W^\perp S \cap R ) \simeq \Bbbk[r_0]$.
Similarly
$A/(W^\perp S \cap A ) \simeq \Bbbk[a_{i_0}]$.
Hence,
$\hilbertIdeal_{G,S} = \hilbertIdeal_{G,S,W } + (r_0 )$
and
$\hilbertIdeal_{G',S} = \hilbertIdeal_{G',S,W } + (a_{i_0})$.
\proofstep
For a homogeneous $S$-ideal $J$, write $\deg (S/J)$ for the \define{degree}
(or \define{multiplicity}) of $S/J$.
Then, by Lemma~\ref{lemma:a3notinR},
\begin{align*}
|G| = \dim_\Bbbk S/\hilbertIdeal_{G,S }
& = \deg r_0 \deg(S/\hilbertIdeal_{G,S,W})
\\
& = p \deg (a_{i_0}) \deg(S/\hilbertIdeal_{G',S,W})
= p \dim_\Bbbk S/\hilbertIdeal_{G',S }
\geq p |G'| = |G|,
\end{align*}
so $\dim_\Bbbk S/\hilbertIdeal_{G',S } = |G'|$.
In other words,
\[
\dim_\Bbbk S \otimes_A A/A_+ =
\dim_{\Frac(A)} S \otimes_A \Frac(A)
\]
so, by the Nakayama lemma, $S$ is a free $A$-module.
By Discussion~\ref{discussionbox:free}, $A$ is a polynomial ring.
\end{steps}
\end{proof}

\section{Hilbert ideals in dimension $4$}
\label{section:hilb4}
In the following we assume $\dim_\Bbbk V=4$ and we adopt the notation of
Lemma \ref{lemma:algGens}.
Let $W=\Bbbk\langle v_4\rangle\subseteq V^G$.

\begin{lemma}\label{lemma:a3inR}
Suppose that $R$ is a direct summand of $S$
and that $\beta_\sigma > 1$.
If $a_3\in R$ and $\hilbertIdeal_{G', S, W}$ is a complete
intersection, then
\[
\hilbertIdeal_{G', S, W} =
\hilbertIdeal_{G, S, W} = (x_1, a_2, a_3)S.
\]
\end{lemma}

\begin{proof}
It suffices to show that
\begin{equation}
\label{equation:a3inRHilbIdeal}
\hilbertIdeal_{G', S, W} = (x_1, a_2, a_3)S,
\end{equation}
for, then, it would follow that
\[
\hilbertIdeal_{G', S, W}=(x_1,a_2,a_3)S\subseteq
\hilbertIdeal_{G, S, W}\subseteq \hilbertIdeal_{G', S, W},
\]
proving the proposition.
(Note that $a_2 \in R$ by Lemma~\ref{lemma:secondgen}.)

To prove~\eqref{equation:a3inRHilbIdeal}, in turn, it suffices to show that
$x_1, a_2, a_3$ are part of a minimal generating set of
$\hilbertIdeal_{G', S, W}$, i.e., none of them belong to
$S_+ \hilbertIdeal_{G', S, W}$.
Note that $ \height \hilbertIdeal_{G', S, W} = 3$.
Since $\hilbertIdeal_{G', S, W}$ is a complete
intersection,~\eqref{equation:a3inRHilbIdeal} would then follow.

It is immediate that $x_1\notin S_+\hilbertIdeal_{G', S, W}$.
Suppose that $a_2 \in S_+\hilbertIdeal_{G', S, W}$.
Since $a_2 \not \in x_1S$, for $a_2$ to belong to
$S_+\hilbertIdeal_{G', S, W}$, it is necessary that there is a
homogeneous element $b$ of $S^{G'}$ such that $b \not \in \Bbbk[x_1]$
and $\deg b < \deg a_2$. This contradicts the choice of $a_2$.
Hence $a_2 \not \in S_+\hilbertIdeal_{G', S, W}$.

We now show that $a_3 \notin (x_1,a_2)S$.
Since $R$ is a direct summand of $S$ and $a_3\in R$, if $a_3\in (x_1,a_2)S$
then $a_3\in (x_1,a_2)R$.
This implies $a_3\in (x_1,a_2)A$.
Write $a_3=a'x_1+a''a_2$ with $a', a''\in A$.
By degree considerations, we have that $a', a''\in \Bbbk[x_1,a_2]$ which is
a contradiction of the fact that $\{x_1, a_2, a_3\}$ is a part of minimal
generating set of $A$ as a $\Bbbk$-algebra.
Hence $a_3 \notin (x_1,a_2)S$.

Now suppose that $a_3 \in S_+\hilbertIdeal_{G', S, W}$.
Since $a_3 \notin (x_1,a_2)S$,
it is necessary that there is a homogeneous element $b$ of
$S^{G'}$ such that $b \not \in (x_1, a_2 )S$ and $\deg b < \deg a_3$.
Note that $b \not \in \Bbbk[x_1, a_2 ]$.
This contradicts the choice of $a_3$.
Hence $a_3 \not \in S_+\hilbertIdeal_{G', S, W}$.
\end{proof}

\begin{lemma}\label{lemma:betaGprime}
  Suppose $\Bbbk=\mathbb{F}_p$. Then $\beta_{G'}\leq 2$. In particular $\hilbertIdeal_{G', S, W}$ is a complete intersection.
\end{lemma}
\begin{proof}
If $\beta_{\sigma}=2$, then by the construction of the filtration
$\beta_{\tau}\leq 2$ for all pseudoreflections $\tau\in G$.
Therefore assume $\beta_{\sigma}=3$.
By way of contradiction, if $x_3$ appears in $(\tau-1)x_4$, then for some
$k=0,\ldots, p-1$, we have $\beta_{\sigma^k\tau}\leq 2$.
By our choice of filtration, $\sigma^k\tau\in G'$, which is a
contradiction.
Hence the claim follows.
The second statement follows from Proposition
\ref{proposition:largeRankVGRel} since $\beta_{G}\leq 2$ implies
$\dim_\Bbbk V^{G}\geq 2=\dim_\Bbbk V-2$.
\end{proof}

\begin{lemma}
\label{lemma:ai0}
Adopt the notation from Lemma~\ref{lemma:algGens}.
Assume that $\beta_{G'} \leq 2$, $i_0=3$ (i.e., $a_3 \not \in R$),
and $\beta_\sigma = 3$.
Then
\begin{enumerate}
\item
\label{lemma:ai0:diff}
$\In_> ((\sigma -1) a_{3 }) = x_3^{\deg a_{3}}$;

\item
\label{lemma:ai0:noprod}
No term of $a_{3}$ is divisible by $x_3x_4$, i.e., coefficients of
$x_4^t$ in $a_{3}$ are in $\Bbbk[x_1,x_2]$ for all $t> 0$.

\item
\label{lemma:ai0:in}
$\In_> ( a_{3 }) = x_4^{\deg a_{3}}$, i.e., $i_1=3$.

\end{enumerate}
\end{lemma}

\begin{proof}
\eqref{lemma:ai0:diff}:
By~\cite[Theorem~4(i)]{BroerHypersurfaces2006},
$((\sigma -1) a_{3 })^{p-1}$ generates the Dedekind different ideal
$\calD_{A/R}$.
Fix generators $\Delta_{S/R}$ and $\Delta_{S/A}$ for the Dedekind
differents $\mathcal{D}_{S/R}$ and $\mathcal{D}_{S/A}$ respectively.
Then
\[
\Delta_{S/R} = \Delta_{S/A}  \cdot ((\sigma -1) a_{3 })^{p-1}
\]
up to multiplication by an element of $\Bbbk^\times$.

For each transvection $\tau \in G$, fix a linear form $l_\tau$
that defines its reflecting hyperplane; $l_\tau$ is defined uniquely up to
multiplication by elements of $\Bbbk^\times$.
Hence $\In_> l_\tau = x_3$ if and only if $\tau \in G \minus G'$.
Moreover, the irreducible factors of
$\Delta_{S/R}$ (respectively, $\Delta_{S/A}$) are precisely the $l_\tau$
with $\tau$ a transvection in $G$ (respectively, $G'$);
see~\cite[Proposition~3.10.3]{BensonInvBook1993}
or~\cite[Proposition~9, p.~17]{BroerDirectSummandProperty2005}.
Therefore the irreducible factors of $(\sigma -1) a_{3 }$ are precisely
the $l_\tau$ with $\tau$ a transvection in $G \minus G'$.
This proves~\eqref{lemma:ai0:diff}.

\eqref{lemma:ai0:noprod}:
Since $\beta_{G'}\leq 2$, $x_3$ is a terminal variable for $G'$.
Hence, for all $t>0$,
$\Delta_{x_3}^{(t)}(a_{3})$ belongs to $A$ and, hence, to $\Bbbk[x_1,
a_2]$, by degree considerations.
If $x_3x_4$ divides some term of $a_{3 }$, then there exists $t>0$ such
that $x_4$ divides a term of $\Delta_{x_3}^{(t)}(a)$.
This is a contradiction since $\{x, a_2 \} \subseteq \Bbbk[x_1, x_2, x_3]$.

\eqref{lemma:ai0:in}:
For all $a,b,c \in \naturals$ with $c \neq 0$,
If $\In_>(a_{3 } ) = x_1^ax_2^bx_4^c$, then
$\In_>((\sigma-1 )a_{3 } ) = x_1^ax_2^bx_3^c$.
Moreover for all $a,b,c \in \naturals$,
$(\sigma-1 )(x_1^ax_2^bx_3^c ) = 0$.
Hence the assertion follows from \eqref{lemma:ai0:diff} and
\eqref{lemma:ai0:noprod}.
\end{proof}

\begin{proof}[Proof of Theorem~\protect{\ref{theorem:SWdim4primefield}}]
We show that $\hilbertIdeal_{G,S}$ is a complete intersection.

\underline{Case I: $\beta_G \leq 2$}.
Then $\hilbertIdeal_{G,S,V^G}$ is a complete
intersection (Proposition~\ref{proposition:largeRankVGRel}).
This case does not use the hypothesis that $\Bbbk = \bbF_p$.

\underline{Case II: $\beta_G = \beta_\sigma = 3$}.
Let $W = \Bbbk\langle v_4 \rangle$.
Apply Lemma~\ref{lemma:betaGprime} to see that
$\hilbertIdeal_{G',S,W}$ is a complete intersection.
By Lemma~\ref{lemma:secondgen}, $i_0 \geq 3$.
If $i_0=3$, then Lemmas~\ref{lemma:betaGprime}, \ref{lemma:ai0}
and~\ref{lemma:a3notinR} imply that
$\hilbertIdeal_{G,S,W}$ is a complete intersection.
If $i_0 > 3$, then Lemma~\ref{lemma:a3inR} implies that
$\hilbertIdeal_{G,S,W}$ is a complete intersection.

Finally use
Proposition~\ref{proposition:Broer22}\eqref{proposition:Broer22:mingens}
to see that $\hilbertIdeal_{G,S}$ is a complete intersection.
\end{proof}

\section{Proof of Theorem~\protect{\ref{theorem:smallG}}}
\label{section:smallG}

The following lemma might be known, but we give a proof for the sake of
completeness.

\begin{lemma}
\label{lemma:psquare}
Let $G$ be a transvection group of order $p^2$.
Then $S^G$ is a polynomial ring.
\end{lemma}

\begin{proof}
Every group of order $p^2$ is abelian. Let $\sigma, \tau \in G$ be
transvections that generate $G$. Fix equations $l_\sigma$ and $l_\tau$
defining the reflecting hyperplanes of $\sigma$ and $\tau$ respectively.
Note that $\sigma \tau \sigma^{-1 }$ is a transvection, with reflecting
hyperplane defined by $\sigma l_\tau$; similarly,
$\tau \sigma \tau^{-1 }$ is a transvection, with reflecting
hyperplane defined by $\tau l_\sigma$.
Since $\sigma \tau = \tau \sigma$, we see that $\sigma l_\tau = l_\tau$; in
any case, $\tau l_\tau = l_\tau$. Therefore $l_\tau \in S^G$. Similarly
$l_\sigma \in S^G$.

If $l_\tau$ and $l_\sigma$ are linearly dependent over $\Bbbk$, then $G$
has exactly one reflecting hyperplane. Therefore
using~\cite[Theorem~5, Appendix]{LandweberStongDepth1987}
we see that $S^G$ is a polynomial ring.

Otherwise, we could assume that $l_\sigma = x_1$ and $l_\tau = x_2$.
After permuting the variables $x_3, \ldots, x_n$ and replacing $x_1$
by a suitable scalar multiple of it,
we may assume that $\sigma x_3 = x_3 + x_1$. For all $i \geq 4$, after
replacing $x_i$ by $x_i + \lambda x_3$ for a suitable $\lambda \in \Bbbk$,
we may assume that $\sigma x_i = x_i$.
Similarly, if there exists $i \geq 4$ such that $\tau x_i \neq x_i$, we may
assume that $\tau x_4 \neq x_4$ and that that $\tau x_i = x_i$ for all $i
\geq 5$.
Hence $S^G = \Bbbk[x_1, \ldots, x_4 ]^G \otimes_\Bbbk
\Bbbk[x_5, \ldots, x_n]$.
Therefore we may assume that $n=4$.
Further replacing $x_3$ with $x_3 + \lambda x_4$ for a suitable $\lambda$,
we may assume that $\tau x_3 = x_3$. Hence
$S^G$ contains $x_1, x_2, \prod_{i=0}^{p-1} \sigma^i x_3,
\prod_{i=0}^{p-1} \tau^i x_4$
which satisfy the statement (b) of
\cite[Theorem~3.9.4]{DerksenKemperComputationalInvariantThy2ed2015}.
Hence $S^G$ is a polynomial ring.
\end{proof}

\begin{proposition}
Assume that $|G | = p^3$. If $S^G$ is a direct summand of $S$, then $S^G$
is a polynomial ring.
\end{proposition}

\begin{proof}
Note that $G$ is a transvection group.
Let
$0  = G_0 \subsetneq G_1 \subsetneq G_2 \subsetneq G_3 = G$
be the filtration from Proposition \ref{proposition:ourcompositionseries}.
For $l=1,2$, $S^{G_l }$ is a direct summand of $S$
by~\cite[Theorem~5, Appendix]{LandweberStongDepth1987}
and
Lemma~\ref{lemma:psquare}, respectively.
Now apply Proposition~\ref{proposition:allSplit}.
\end{proof}

\section{Examples}
\label{section:examples}

The first example shows that $\hilbertIdeal_{G,S,W } \neq
\hilbertIdeal_{G,\Sym W^\perp }S$, in general.
This is related to Proposition~\ref{proposition:extended}.
The other two examples are
related to Proposition~\ref{proposition:ourcompositionseries}.

\begin{example}
	\label{example:nothGR}
	Let $\Bbbk=\mathbb{F}_9$ and choose $a\in \Bbbk\setminus \mathbb{F}_3$. Consider the action of $G=\langle \tau, \sigma\rangle$ on a $3$-dimensional vector space $V=\Bbbk\langle v_1,v_2,v_3\rangle$ such that the action on $V^*$ with dual basis $\{x, y, z\}$ is given by
	$\tau(x)=x, \tau(y)=y+ax, \tau(z)=z$ and $\sigma(x)=x, \sigma(y)=y+x, \sigma(z)=z+x$. Write $S=\Bbbk[x,y, z]$ and $W=\Bbbk\langle v_3\rangle\subset V^G$.
	We get
	$S^G=\Bbbk[x, (a-1)y^3-ayx^2+zx^2, z^3-x^2z]$, hence
	$\hilbertIdeal_{G,S,W }=(x, y^3)$.
Let $S':=\Sym W^\perp=\Bbbk[x,y]$. Then $(S')^G$ is
generated by $x$ and the $G$-orbit product of $y$, which is a
polynomial of degree $9$.
Hence $\hilbertIdeal_{G,S'}=(x, y^9) \neq
\hilbertIdeal_{G,S,W} \cap S'$.
\end{example}

\begin{example}\label{example:NakajimaFiltration}
	Let $G$ be a Nakajima group with Nakajima basis $\{x_1,\ldots,x_n\}$ of $V^*$. We shall construct a filtration of the form appearing in Proposition \ref{proposition:ourcompositionseries} (with respect to the basis $\{x_1,\ldots,x_n\}$), such that each $G_i$ in the filtration is a Nakajima group with respect to the same basis. We have
	$$G=P_nP_{n-1}\cdots P_1:=\{g_ng_{n-1}\cdots g_1\mid g_k\in P_k\}$$
where $P_k=\{g\in G\mid gx_j=x_j \text{ for all }j\neq k\}$.

For any subgroup $H$ of $G$, let us define $P_k(H)=P_k\cap H$. Then $H$ is
a Nakajima group with same Nakajima basis as $G$, if
$H=P_n(H)P_{n-1}(H)\cdots P_1(H)$. First we show that each $G_l=\{g\in
\calP\mid \beta_g\leq l\}$ is a Nakajima group. Let $g\in G_l$ and write
$g=g_ng_{n-1}\cdots g_1$ with $g_k\in P_k$. Since $g(x_k)=g_k(x_k)$, we
must have $\beta_{g_k}\leq l$, i.e., $g_k\in P_k(G_l)$. Hence $G_l$ is a
Nakajima group.

Now we describe a filtration as in Proposition
\ref{proposition:ourcompositionseries} such that each $G_i$ in the
filtration is a Nakajima group. Note that classes of the elements of the
set $\bigcup_{k=l+1}^nP_k(G_{l+1})\setminus P_k(G_l)$ generate the group
$G_{l+1}/G_l$. Let $l+1< s\leq n$ be the smallest integer such that
$P_s(G_{l+1})\setminus P_s(G_l)\neq \varnothing$. Set $G_{l, 1}=\langle
G_{l}, g\rangle$ for some $g\in P_s(G)$ with $\beta_g=l+1$. Inductively,
define a filtration
$$G_l< G_{l, 1}<\cdots<G_{l,l_0}<G_{l+1}$$
such that if $G_{l, i}\cap \big(P_k(G_{l+1})\setminus P_k(G_l)\big)\neq \varnothing$, then $P_{k-1}(G_{l+1})\setminus P_{k-1}(G_l)\subset G_{l, i}$ for all $k>l+1$ and that each quotients in the above filtration is a group of order $p$.

We claim that each $G_{l,i}$ in the above filtration is a Nakajima group.
For $g\in G_{l,i}$, write $g=g_ng_{n-1}\cdots g_1$ with $g_k\in P_k$.  We
can also write $g=h_1\tau_1\cdots h_r\tau_r$ where $h_j\in G_l$ and
$\tau_j\in \big\langle G_{l,i}\cap
\big\{\bigcup_{k=l+1}^{k_0}P_k(G_{l+1})\setminus P_k(G_l)\big\}\big\rangle
$. Suppose $l+1\leq k_0\leq n$ is the largest integer such that $G_{l,
i}\cap \big(P_{k_0}(G_{l+1})\setminus P_{k_0}(G_l)\big)\neq \varnothing$.
If $k_0<k_1\leq n$, then we have $g(x_{k_1})=h_1h_2\cdots h_r(x_{k_1})$.
Hence $\beta_{g_{k_1}}\leq l$ which implies $g_k\in P_k(G_l)$ for all
$k>k_0$. Since $\beta_{g_k}\leq l+1$, by definition of $G_{l,i}$, we have
$g_k\in P_k(G_{l,i})$ for all $k<k_0$. This forces $g_{k_0}\in G_{l,i}$.
Hence each $G_{l,i}$ is a Nakajima group.
\end{example}

\begin{example}
	\label{example:notInertia}
Let $\Bbbk$ be a field of characteristic $3$ and $V$ is a $\Bbbk$-vector
space of dimension $5$. Let $G$ be  a $3$-group generated by $\langle
g_1,g_2,g_3\rangle$ with actions on $V^*=\Bbbk\langle
x_1,\ldots,x_5\rangle$ given as follows:
\begin{align*}
g_1(x_i)=x_i \;\text{for}\; & i\neq 4 \;\text{and}\; g_1(x_4)=x_4+x_3+x_2;
\\
g_2(x_i)=x_i \;\text{for}\; & i\neq 5 \;\text{and}\; g_2(x_5)=x_5+x_3-x_2+x_1;
\\
g_3(x_i)=x_i \;\text{for}\; & i\neq 4, 5 \;\text{and}\;
g_3(x_4)-x_4=g_3(x_5)-x_5=-x_3-x_2-x_1.
\end{align*}

For a prime ideal $\frakq$ of $S$, write $I_\frakq$ for its inertia
subgroup, $\{ g \in G \mid (g-1)s \in \frakq ;\text{for all}\; s \in S \}$.
We use this representation to show that the filtration in Proposition
\ref{proposition:ourcompositionseries} need not be same as the filtration
obtained by refining the filtration by the inertia groups of the ideals
$(x_1,\ldots,x_i)$, i.e., the group generated by $\{g\in G\mid \beta_g\leq
i\}$.
We can check that $I_{(x_1, x_2)}$ is
a cyclic group of order $3$ generated
by $g_1g_2g_3$ which is not a pseudoreflection. Hence no refinement of the
filtration $I_{(x_1, x_2)} \subsetneq I_{(x_1,x_2,x_3)} = G$
will give us a filtration of the type described in Proposition
\ref{proposition:ourcompositionseries}.
\end{example}

\ifreadkumminibib
\bibliographystyle{alphabbr}
\bibliography{kummini}
\else

\newcommand{\etalchar}[1]{$^{#1}$}
\def\cfudot#1{\ifmmode\setbox7\hbox{$\accent"5E#1$}\else
  \setbox7\hbox{\accent"5E#1}\penalty 10000\relax\fi\raise 1\ht7
  \hbox{\raise.1ex\hbox to 1\wd7{\hss.\hss}}\penalty 10000 \hskip-1\wd7\penalty
  10000\box7}

\fi

\end{document}